\documentclass[11pt, final]{amsart} 
\usepackage{amsmath,amssymb, amsthm,amsbsy, amsfonts,amsthm, amscd, amssymb,amscd,mathrsfs, showkeys} 
\usepackage[dvips]{graphicx}
\usepackage{epsfig}
\usepackage[subnum]{cases}
\usepackage[colorlinks=true, linkcolor=magenta, citecolor=cyan, urlcolor=blue,hyperfootnotes=false]{hyperref}
\newcommand{\R}{{\mathbb R}}
\newcommand{\N}{{\mathbb N}} 
\newcommand{\T}{{\mathbb T}} 

\newcommand{\C}{{\mathbb C}}

\newcommand{\Lin}{\mathcal{L}}
\newcommand{\Hin}{\mathcal{H}}
\newcommand{\Win}{\mathcal{W}}

 \renewcommand{\geq }{\geqslant}
 \renewcommand{\leq }{\leqslant}

\usepackage{mathtools}
\DeclarePairedDelimiter{\abs}{\lvert}{\rvert}

\DeclarePairedDelimiter{\norma}{\lVert}{\rVert} 
\usepackage{braket}
\usepackage{color} 
\newcommand{\Rn}{{\mathbb R^n}}
\newenvironment{sistema}%
{\left\lbrace\begin{array}{@{}l@{}}}%
{\end{array}\right.}  

\numberwithin{equation}{section}
\newtheorem{theorem}{Theorem}[section]
\newtheorem{corollary}[theorem]{Corollary}
\newtheorem{lemma}[theorem]{Lemma}
\newtheorem{proposition}[theorem]{Proposition}
\theoremstyle{definition} 
\newtheorem{definition}[theorem]{Definition}

\newtheorem{remark}[theorem]{Remark}

\begin{document} 

\title[$H^2$-scattering for fourth order NLS]{$H^2$-scattering for systems of weakly coupled fourth-order NLS equations in low space dimensions}


\author{M.~Tarulli}
\address{Faculty of Applied Mathematics and Informatics, Technical University of Sofia, Kliment Ohridski Blvd. 8, 1000 Sofia and IMI--BAS, Acad. Georgi Bonchev Str., Block 8, 1113 Sofia, Bulgaria}
\email{mta@tu-sofia.bg}



\begin{abstract}
We prove large-data scattering and existence of wave operators in the energy space for the systems of $N$ defocusing fourth-order Schr\"odinger equations with mass-supercritical and energy-subcritical power-type nonlinearity. In addition,
new nonlinear interaction Morawetz identities and inequalities are given, suitable to shed lights on the decay of the solution with respect some Lebesgue norms when the space dimensions are $d=3,4$.
\end{abstract}

\maketitle

\section{Introduction}\label{sec:introduction}

The main target of the paper is the analysis of the decaying and scattering properties of the solution to the following system of $N\geq 1$ defocusing nonlinear fourth-order Schr\"odinger equations (NL4S) in dimension $d \geq 3 $:
\begin{equation}\label{eq:nls}
\begin{cases}
i\partial_t u_\mu + (\Delta^2-\kappa\Delta )u_\mu +
\displaystyle{\sum_{\substack{\nu=1 }}^N}
G(u_\mu, u_\nu)=0, 
\\
(u_\mu(0,\cdot))_{\mu=1}^N= (u_{\mu,0})_{\mu=1}^N \in H^2(\R^d)^N,
\end{cases}
\end{equation}
\begin{equation}\label{eq.nonlst}
G(u_\mu, u_\nu)=
\beta_{\mu\nu}|u_\nu|^{p+1}|u_\mu|^{p-1}u_\mu+\displaystyle{\sum_{\substack{\mu=1 }}^N}
\lambda_{\mu\nu}|u_\nu|^{p+1}|u_\mu|^{p-1}u_\mu
\end{equation}
and $\kappa=0,1$. Here, for all $\mu,\nu=1,\dots,N$, $u_\mu=u_\mu(t,x):\R\times\R^d\to\C$, $(u_\mu)_{\mu=1}^N=(u_1,\dots, u_N)$ and  $\beta_{\mu\nu}, \lambda_{\mu\nu}\geq0$, with either $\beta_{\mu\mu}\neq 0$ or $\lambda_{\mu\mu}\neq 0$,
 are coupling parameters, as well we require that the nonlinearity parameter $p$ is constrained to the following conditions 
\begin{align}\label{eq:base}
&1 \leq p<p^*(d), \ \ \ \ pd>4,\ \ \ \  p^*(d)=
\begin{cases}
+\infty \ \ \ \ \ \, & \text{if} \ \ \ d=3,4,\\
\frac 4{d-4}  \ \ \ \  \  \  &\text{if} \ \ \ 5\leq d\leq 8.
\end{cases} 
\end{align}
Furthermore, the power nonlinearity $p^*(d)$ is the  $H^2$-critical exponent for the single NL4S in $\R^{d}$, whereas the lower bound $\max(1,  4/d)$ relies on some restrictions due to the well-posedness of \eqref{eq:nls} in the product space $H^2(\R^d)^N$, as we see afterwards in the Remark \ref{beta0}. 
We recall also that there are two important conserved quantities of the system \eqref{eq:nls}. Namely, the mass
\begin{align}\label{eq.mass}
  M(u_\mu)(t)=\int_{\R^d}|u_\mu(t)|^2\,dx,
\end{align}
for $\mu=1,\dots,N$ and the energy  
 \begin{equation}\label{eq:energy}
\begin{split}
  E(u_{1},\dots,u_{N})
  \\
 = \int_{\R^d} \sum_{\mu=1}^N \abs{\Delta{u_{\mu}}}^2
 +\kappa \int_{\R^d} \sum_{\mu=1}^N \abs{\nabla{u_{\mu}}}^2
 +\sum_{\mu,\nu=1 }^N
  (\beta_{\mu\nu}+N\lambda_{\mu\nu}) \int_{\R^d}\frac{|u_\mu u_\nu|^{p+1}}{p+1} \,dx.&
  \end{split}
\end{equation}

The fourth-order Schr\"odinger equations are important in several models of mathematical physics. Introduced in \cite{FIP} to describe small dispersion in the propagation of intense laser beams in a medium with Kerr nonlinearity, it was successively used in the context of the theory of motion of a vortex filament in an incompressible fluid in \cite{Kar}, \cite{KaS}, \cite{Se03}, see also \cite{HJ05} and \cite{HJ07}. Motivated by this, here we investigate large-data scattering in $H^2(\R^d)^N$ for \eqref{eq:nls}, in parallel with the case of the single defocusing NL4S
\begin{equation}\label{eq:nle}
  \begin{sistema}
    i \partial_t u + (\Delta^2-\kappa\Delta)u + \abs{u}^{2p}u=0,\\
    u(0)=u_0 \in H^2(\R^d),
  \end{sistema}
\end{equation}
with $u : \R \times \R^d \to \C$, $\kappa=0,1$ and $p>0$, following the seminal ideas unfolded in the papers \cite{CT} for systems of NLS, \cite{Vis} for a simple NLS and proposing some relevant novelties. To be specific, in a first step we carry out new Morawetz identities, interaction Morawetz identities and their corresponding inequalities for \eqref{eq:nls} extending the proofs given in \cite{GS}. Then as a second step, via the localization of the nonlinear part of Morawetz inequalities on space-time slabs, with the space components being chosen as $\R^d$-cubes, we can perform a contradiction argument which enables us to show the decay of $L^q$-norms of the solutions to \eqref{eq:nle} as $t\rightarrow \pm \infty$, as long as $2<q<2d/(d-4)$, for $5\leq d\leq 8$ and $2<q<\infty$ for $d=3,4$ in analogy to \cite{CT} and  \cite{Vis}, actually easing in our framework the proof of such a phenomenon. This particular effect, in combination with a generalization of the theory developed in \cite{Ca} to the systems of NL4S, infers to have asymptotic completeness and existence of the wave operators in the energy space $H^2(\R^d)^N$ for solution to \eqref{eq:nls}. We point out that our result relies on an approach that displays the asymptotics without making a distinction between the number of coupled equations. Indeed, the interaction Morawetz estimates are presented in a suitable form which allows to deal only with its nonlinear part, letting to a new class of correlation-type inequalities. In such a way, by this new kind of nonlinear estimates, the possibility to overcome the mathematical obstacles arises, as well as the opportunity to supply a further simple proof of scattering results attained in \cite{Pa1}, \cite{PaX}, where it is studied the asymptotic behaviour of solutions in the subcritical and energy-critical case and in \cite{GS}, in which it is examined for the first time a system of coupled NL4S. In all these papers, the authors produce a
set of linear Morawetz estimates, guaranteeing that the wave and the scattering 
operators for \eqref{eq:nle}  are well-defined and bijective in the energy-space $H^2(\R^d)$, for  $d\geq 5$, but not fitted to cover the small space dimensions framework.
\\

 Now, we state the main result of this paper, that is 

\begin{theorem}\label{thm:main}
  Let be $3\leq d \leq 8$, $p\in\R$ such that \eqref{eq:base} holds,
  then:
  \begin{itemize}
  \item\emph{(asymptotic completeness)} If $(u_{\mu,0})_{\mu=1}^N \in H^2(\R^d)^N,$ then the unique global solution to \eqref{eq:nls} 
    $(u_{\mu})_{\mu=1}^N \in\mathcal C(\R, H^2(\R^d)^N)$, for $\kappa=0,1$, 
    scatters, i.e. there exist $(u_{\mu,0}^{\pm})_{\mu=1}^N \in H^2(\R^d))^N$ such that for all $\mu=1,\dots,N$ 
    \begin{equation}\label{eq:scattering}
      \lim_{t\to\pm\infty}\left\|u_{\mu}(t,\cdot)-e^{it(\Delta^2-\kappa\Delta)}u_{\mu,0}^\pm(\cdot)\right\|_{H^2}=0.
    \end{equation} 
  \item\emph{(existence of wave operators)} For every $(u_{\mu,0}^{\pm})_{\mu=1}^N \in H^2(\R^d)^N$ there exist unique 
    initial data $(u_{\mu,0})_{\mu=1}^N \in H^2(\R^d)^N,$ such that the
    global solution to \eqref{eq:nls} 
    $(u_{\mu})_{\mu=1}^N \in\mathcal C(\R, H^2(\R^d)^N)$  satisfies
    \eqref{eq:scattering}.
  \end{itemize}

\end{theorem}

\begin{remark}\label{beta0} 
We observe, as aforementioned, that the study of the systems is forbidden in dimension $d\geq 8$
since a lack of an existence theorem like Proposition \ref{ConsLaw}. More precisely, if for $\mu\neq\nu$, some of the parameters $\beta_{\mu\nu}$ and $\lambda_{\mu\nu}$  are not vanishing in \eqref{eq:nls}, we 
  are obligated to assume $p\geq1$ because of the structure of the nonlinearity \eqref{eq.nonlst}. Anyway, in the simple case $\beta_{\mu\nu}=\lambda_{\mu\nu}=0$ 
  for all $\mu\neq\nu$, we are no longer forced to impose further lower bounds than $p>4/d$ and consequently, as a side effect of this paper, we get decay w.r.t. Lebesgue norms for $0<p<4/(d-4)$ and scattering for $4/d<p<4/(d-4)$ to the solution of the problem \eqref{eq:nle} in all dimensions $d \geq 3$. 
  \end{remark}
  
In light of that we have

\begin{corollary}\label{thm:main2}
  Let be $d\geq 3$, $N=1$ and $\frac4d<p<\frac{4}{d-4}$ 
  then, if $u_0=u_{1,0} \in H^2(\R^d),$ then the unique global solution to \eqref{eq:nls}
  $u=(u_{1}) \in\mathcal C(\R, H^2(\R^d))$, for $\kappa=0,1$,  is such that:
  \begin{itemize}
  \item the decay property,
    \begin{align}\label{eq:decay0}
\lim_{t\rightarrow \pm \infty} \|u(t, \cdot)\|_{L^q}=0,
\end{align}
is fulfilled with $0<p<\frac{4}{d-4}$, $2<q<\frac{2d}{d-4},$ for $d\geq 5$, and with $0<p<+\infty$, $2<q<+\infty,$ for $d=3, 4$;
    \item
   if $\frac4d<p<\frac{4}{d-4}$, for  $d\geq 5$ and $\frac4d<p<+\infty$, for $d=3, 4$ the scattering occurs, i.e. there exist $u_0^{\pm}=u_{1,0}^{\pm} \in H^2(\R^d))$ such that  
    \begin{equation}\label{eq:scattering0}
      \lim_{t\to\pm\infty}\left\|u(t,\cdot)-e^{it(\Delta^2-\kappa\Delta)}u_{0}^\pm(\cdot)\right\|_{H^2}=0.
    \end{equation}
    \end{itemize}
\end{corollary}
\begin{remark}\label{1eqscatt}
Considering the decay of the $L^q$-norm, we underline that the above \eqref{eq:decay0}, formulated in Theorem \ref{thm:main2}, were originally established in \cite{PaX}, but with limitation to $p>\frac 4d$ (for small-data setting also). We emphasize here that our technicalities extend the result to the range to $p\leq \frac 4d$ which represents the novelty for the case of a single NL4S. For what regards the scattering result \eqref{eq:scattering0}, it appears in \cite{PaX}, for the dimensions $1\leq d\leq 4$. However a techniques from Kenig and Merle (see \cite{KM}) in combination with virial-type ingredient is employed to balance the absence of classical Morawetz-type estimates. Our method is an alternative and easier way to achieve the same results in dimensions $d=3,4$, by using new 
nonlinear Morawetz-type inequalities generalized to the systems background.
\end{remark}

Looking at the literature, we end shortly by recalling some of the known general achievements linked to the problem \eqref{eq:nle}, either for $\kappa=0$ or $\kappa=1$, with the recommendation to look at references therein.
In \cite{BKS00} dispersive estimates for the biharmonic Schr\"odinger operator are given  which infer to the Strichartz estimates for the fourth-order Schr\"odinger equations. In \cite{Se06} the author displays one dimensional modified scattering for cubic nonlinearity.  The paper \cite{PaSh} contains the scattering analysis for the mass critical case in high dimensions. We refer also to the important work \cite{PaX} (see Remark \ref{1eqscatt}), besides the already cited papers \cite{Pa1}, \cite{Pa2}, while as far as we know scattering results are not available in the systems set-up with the exception of the also mentioned paper \cite{GS}. Additionally, we look back to the fact that the Morawetz multiplier method and the resulting estimates are mandatory to examine scattering properties for other nonlinear dispersive equations. These estimates were achieved initially in \cite{Morawetz} for the nonlinear Klein-Gordon equation and then used for retrieving the asymptotic completeness of the NLS in various paper, for example we remand to \cite{GiVel2} \cite{LinStrauss}. In the recent period the interaction Morawetz estimates, a powerful approach that is  based on getting bilinear Morewetz inequalities, played a crucial role in simplifying the proof of scattering. We quote in this direction the papers  \cite{CGT}, \cite{CKSTT1}, \cite{CKSTT2}, \cite{GiVel} and the remarkable paper
\cite{PlVe} where the interaction Morawetz estimates involving only the Laplacian of the Morawetz multipliers appear for the $L^2$-supercritical and $H^1$-subcritical NLS. We quote also the paper \cite{TzVi}, where the interaction Morawetz technique is extended to turn out the the scattering in $H^1$ for NLS posed on the product-type geometry $\R^d\times\T,$ with $d\geq1$.

\subsection*{Outline of paper.} Along Section \ref{InterMor} we set-up, in Lemma \ref{lem:mor} and Lemma \ref{lem:intmor},
the interaction Morawetz identities 
and inequalities respectively and  in Propositions \ref{dim3}, the connected nonlinear Morawetz estimates
for the system of NL4S \eqref{eq:nls}, fundamental tools for proving Theorem \ref{thm:main}. 
The Section \ref{MainThm} is splitted in two part: in the first we show how the interactive Morawetz inequalities 
enable to exploit the decay of some Lebesgue-norms of the solutions to \eqref{eq:nls}, this is included in Proposition \ref{decay}, having its peculiar interest; in the second we explore the existence of scattering states and wave 
operators by using the result acquired in the first part, completing the proof of Theorem \ref{thm:main}.

\subsection*{Notations.}
We recall that $1\leq r' \leq \infty$ is the H\"older conjugate exponent of any given $1\leq r\leq\infty$. We denote by $L_x^r$ the Lebesgue space $L^r(\Rn)$, and respectively by $W^{2,r}_x$ and  $H^2_x$ the  inhomogeneous Sobolev spaces $W^{2,r}(\Rn)$ and $H^2(\Rn)$ (for more details see \cite{Ad}).
We introduce, for $N\in \N $,  the Lebesgue space $ \Lin_x^r= L^r(\Rn)^N$ and   
the Sobolev spaces by $\Win^{2,r}_x=W^{2,r}(\Rn)^N$ and $\Hin^2_x=H^2(\Rn)^N$, respectively.  We also utilize the symbol $\mathcal D_x$ (resp. $\mathcal D_y$) to make unambiguous the dependence w.r.t. $x$ (resp. $y$) variable of a general differential operator $\mathscr D$.

\section{Morawetz and interaction Morawetz identities}\label{InterMor}
The main aim of this section is to pursue the basic tools for the proof of our main theorem:
the Morawetz-type identities, that are close to the ones holding for the single NLS4. 
We find necessary to introduce the following notations: for any given function $f\in H^2(\R^d,\C)$, we denote by
\begin{equation}\label{eq:notation}
  m_f(x):=|f(x)|^2,
  \qquad
   j_f(x):=\Im\left[\overline{f}\nabla f(x)\right],
\end{equation}
that are the density mass and the density momentum, respectively. In addition, from now on we drop the variable $t$ for simplicity, expressing it only where needed. We have
\begin{lemma}[Morawetz]\label{lem:mor}
Let $d\geq1$, and $(u_\mu)_{\mu=1}^N \in\mathcal C(\R,H^2(\R^d)^N)$ be a global solution to system \eqref{eq:nls},
 let $\phi=\phi(x):\R^d\to\R$ be a sufficiently regular and decaying function, and introduce the action given by
\begin{equation*}
 M(t)
=2\sum_{\mu=1}^N \int_{\R^d} j_{u_\mu}(x)  \cdot\nabla\phi(x)\,dx.
\end{equation*}
The following identity holds:
\begin{align}\label{eq:mor2}
\dot M(t)&
\\
=\sum_{\mu=1}^N\int_{\R^d} m_{u_\mu}(x)(-\Delta^3\phi(x)+\kappa \Delta^2\phi(x))+2\Delta^2\phi(x)|\nabla u_\mu(x)|^2\,dx&
\nonumber
\\
+\sum_{\mu=1}^N\left[
4 \int_{\R^d}\nabla u_{\mu}(x)D^2(\Delta-\kappa)\phi(x)\cdot\nabla \overline u_{\mu}(x)\,dx
\right]&
\nonumber
\\
-\sum_{\mu=1}^N\left[8\int_{\R^d}D^2 u_{\mu}(x)D^2\phi(x)D^2 \overline u_{\mu}(x)\,dx\right]&
\nonumber
\\
-\frac{2p}{p+1}
\sum_{\mu,\nu=1}^N \gamma_{\mu\nu}\int_{\R^d}|u_{\mu}(x)|^{p+1}|u_{\nu}(x)|^{p+1}\Delta\phi(x)\,dx ,&
\nonumber
\end{align}
with $\gamma_{\mu\nu}=\beta_{\mu\nu}+N\lambda_{\mu\nu}$, for all $\mu,\nu=1,\dots N$, $\kappa=0,1$, where $D^2\phi\in\mathcal M_{n\times n}(\R^d)$ is the hessian matrix of $\phi$, $\Delta^2\phi=\Delta(\Delta\phi)$
and $\Delta^3\phi=\Delta\Delta(\Delta\phi)$ the second and third power of the Laplace operator, respectively.
\end{lemma}
\begin{proof}
 Here we follow the spirit of the paper \cite{CT}. For the case of a single NL4S and with $\kappa=0$ we remand also to \cite{MWZ}.  We prove the identities for a Schwartz solution $(u_\mu)_\mu$, allowing the case 
  $(u_\mu)_{\mu=1}^N \in\mathcal C(\R,H^2(\R^d)^N)$ by a density argument 
  (we refer, for instance, to \cite{Ca} or \cite{GiVel}).
  By an integration by parts and utilizing the equation \eqref{eq:nls}, we have for every fixed $\mu$,
  \begin{align}\label{eq:dausare1}
 2\partial_t \int_{\R^d} j_{u_\mu}(x)\cdot \nabla \phi(x) \,dx & \\
    \qquad =- 2\Im  \int_{\R^d} \partial_t u_{\mu}(x) [\Delta \phi(x)\bar u_{\mu}(x)+2\nabla \phi(x)\cdot \nabla \bar u_{\mu}(x)]\,dx& \nonumber\\
    \qquad=2\Re \int_{\R^d} i \partial_t u_{\mu}(x) [\Delta \phi(x)\bar u_{\mu}(x)+2\nabla \phi(x)\cdot \nabla \bar u_\mu(x)]\,dx&
    \nonumber\\ 
    \qquad=2\Re \int_{\R^d} \Big[-\Delta^2 u_{\mu}(x) +\kappa\Delta u_{\mu}(x)+ 
    \sum_{\nu=1}^N \lambda_{\mu\nu}\abs{u_\nu(x)}^{p+1}\abs{u_\mu(x)}^{p-1} u_\mu(x)\Big]&
    \nonumber\\
    \qquad\qquad\qquad\qquad\cdot[\Delta \phi(x)\bar u_\mu(x)+2\nabla \phi(x)\cdot \nabla \bar u_\mu(x)]\,dx.&
    \nonumber
  \end{align}
  First we get (we refer to \cite{CT}, for example)
  \begin{align}\label{eq:dausare2}
    2\Re\int_{\R^d} \kappa\Delta u_{\mu}(x) [\Delta\phi(x)\bar u_{\mu}(x)+2\nabla \phi(x)\cdot \nabla \bar u_{\mu}(x)]\,dx\\
    =  \int_{\R^d} \kappa\Delta^2 \phi(x) \abs{u_{\mu}(x)}^2\,dx -4\kappa \int_{\R^d} \nabla u_{\mu}(x) D^2 \phi(x)\nabla \bar u_{\mu}(x)\,dx.
    \nonumber
  \end{align}
  In addition we have also
  \begin{align}\label{eq:dausare2a}
    2\Re\int_{\R^d} -\Delta^2 u_{\mu}(x) [\Delta\phi(x)\bar u_{\mu}(x)+2\nabla \phi(x)\cdot \nabla \bar u_{\mu}(x)]\,dx&
    \\
    = - \int_{\R^d} \Delta^2 \phi(x) \Delta\abs{u_{\mu}(x)}^2\,dx + 2\int_{\R^d} \Delta\phi(x) |\nabla u_{\mu}(x)|^2\,dx&
    \nonumber\\
    -2\Re\int_{\R^d} \Delta u_{\mu}(x)\nabla \Delta \phi(x) \cdot \nabla \bar u_{\mu}(x)\,dx +2\Re\int_{\R^d} \nabla  \Delta u_{\mu}(x)\Delta \phi(x) \cdot \nabla \bar u_{\mu}(x)\,dx&
    \nonumber\\
    -8\Re\int_{\R^d} D^2 u_{\mu}(x) D^2 \phi(x)D^2 \bar u_{\mu}(x)\,dx&
    \nonumber\\
    -4\Re\int_{\R^d} D^2 u_{\mu}(x)\nabla D^2 \phi(x) \cdot \nabla \bar u_{\mu}(x)\,dx +2\int_{\R^d}   \Delta \phi(x) |D^2 u_{\mu}|^2\,dx.&
\nonumber
 \end{align}
 By applying now the following
\begin{equation*}
\begin{split}
  -\Re\int_{\R^d} \Delta u_{\mu}(x)\nabla \Delta \phi(x) \cdot \nabla \bar u_{\mu}(x)\,dx +\Re\int_{\R^d} \nabla  \Delta u_{\mu}(x)\Delta \phi(x) \cdot \nabla \bar u_{\mu}(x)\,dx&\\
  =  \int_{\R^d} \nabla u_{\mu}(x)D^2 \Delta \phi(x) \cdot \nabla \bar u_{\mu}(x)\,dx-\int_{\R^d}   \Delta \phi(x) |D^2 u_{\mu}|^2\,dx&
  \end{split}
 \end{equation*}
 and
  $$
 2\Re\int_{\R^d} D^2 u_{\mu}(x)\nabla D^2 \phi(x) \cdot \nabla \bar u_{\mu}\,dx =-\int_{\R^d}\nabla u_{\mu}(x)D^2  \Delta \phi(x) \nabla \bar u_{\mu}(x)\,dx,
 $$
we can see that the r.h.s. of the above identity \eqref{eq:dausare2a} is equivalent to
 \begin{align}\label{eq:dausare2b}
    - \int_{\R^d} \Delta^3 \phi(x) \abs{u_{\mu}(x)}^2\,dx + 2\int_{\R^d} \Delta\phi(x) |\nabla u_{\mu}(x)|^2\,dx&
  \\
    +4\int_{\R^d}\Delta u_{\mu}(x)D^2  \Delta \phi(x) \nabla \bar u_{\mu}(x)\,dx 
    -8\int_{\R^d} D^2 u_{\mu}(x) D^2 \phi(x)D^2 \bar u_{\mu}(x)\,dx.
\nonumber
 \end{align}
Moreover, if we indicate by
$$
X_{\mu\nu}= 2 \Re \int_{\R^d} \abs{u_\nu}^{p+1}\abs{u_\mu}^{p-1} u_\mu(x)
    \cdot[\Delta \phi(x)\bar u_\mu(x)+2\nabla \phi(x)\cdot \nabla \bar u_\mu(x)]\,dx
$$
and 
$$
Y_{\mu\nu}=\frac{2p}{p+1}
\int_{\R^d}|u_{\mu}(x)|^{p+1}|u_{\nu}(x)|^{p+1}\Delta\phi(x)\,dx ,
$$
we claim that
\begin{align}\label{eq:dausare30a}
   X_{\mu\nu}=Y_{\mu\nu}.
\end{align}
In fact, we have the following chain of identities
\begin{equation}
  \begin{split}
  2  \Re  \int_{\R^d} \abs{u_\mu u_\nu}^{p+1}\Delta \phi(x)
    +  2\nabla \phi(x) \cdot \frac{\nabla\abs{u_\mu}^{p+1}}{p+1} \abs{u_\nu}^{p+1}\,dx&\\
    =\, 2 \Re  \int_{\R^d} \abs{u_\mu u_\nu}^{p+1}\Delta \phi(x) 
    +\nabla \phi(x) \cdot \frac{\nabla(\abs{u_\mu}^{p+1}  \abs{u_\nu}^{p+1})}{p+1} \,dx & \\
    =\, 2 \left(1 - \frac{1}{p+1}\right) 
      \Re  \int_{\R^d} \abs{u_\mu u_\nu}^{p+1}\Delta \phi(x)  \,dx,
  \end{split}
  \label{eq:dausare3}
\end{equation}
where in the last equality we used integration by parts. Thus, the above \eqref{eq:dausare30a} infers to the identity
\begin{align}\label{eq:dausare30}
    \sum_{\nu=1 }^N \beta_{\mu\nu}X_{\mu\nu} +  \sum_{\mu,\nu=1 }^N\lambda_{\mu\nu}X_{\mu\nu}
    =   \sum_{\nu=1 }^N \beta_{\mu\nu}Y_{\mu\nu} +  \sum_{\mu,\nu=1 }^N\lambda_{\mu\nu}Y_{\mu\nu}.
\end{align}
Taking in account \eqref{eq:dausare1}, \eqref{eq:dausare2}, \eqref{eq:dausare2a}, \eqref{eq:dausare2b}, \eqref{eq:dausare30} and by a further sum over $\mu=1,\dots,N$, we get the proof of \eqref{eq:mor2} completed.
\end{proof}

By an application of the above lemma, we can now move to the proof of the interaction Morawetz identities. More precisely, we have

\begin{lemma}[Interaction Morawetz]\label{lem:intmor}
  Let $(u_\mu)_{\mu=1}^N \in\mathcal C(\R, H^2(\R^d)^N)$ be a global solution to system \eqref{eq:nls}, let $\phi=\phi(|x|):\R^d\to\R$ be a convex radial function, regular and decaying and so that, for any $f\in \C^d$,
 enjoys
 \begin{equation}\label{eq.bbelow2}
D^2 fD^2\phi(|x|)D^2\overline f\geq C_1\rho_1 (|x|)|\nabla_v^{\bot}f|^2
\end{equation}
 and
 \begin{equation}\label{eq.bbelow}
\nabla fD^2\Delta\phi(|x|)\nabla\overline f\leq -C_2\left(\rho_2 (|x|)|\nabla_v^{\bot}f|^2+\rho_3 (|x|)\left|(v\cdot \nabla f)\frac{v}{|v|^2}\right|^2\right),
\end{equation}
with $v\in\R^d$, $\rho_1 (|x|), \rho_2(|x|), \rho_3 (|x|)>0$,  $C_1, C_2>0$
where 
$\nabla_{v}^{\bot}f=\nabla f- (v\cdot \nabla f)v/|v|^2$. Further, let us denote by $\psi=\psi(x,y):=\phi(|x-y|):\R^{2d}\to\R$ and introduce the action
\begin{equation}\label{eq:intmor1}
  \mathcal M(t) =2\sum_{\mu,\iota=1}^N\int_{\R^d\times\R^d}j_{u_\mu}(x)\cdot\nabla_x\psi(x,y) \, m_{u_\iota}(y)\,dxdy.
\end{equation}
Then the following holds:
  \begin{align}
  \mathcal{ \dot M}(t) \label{eq:intmor2}
   \leq
  2\sum_{\mu,\iota=1}^N\int_{\R^d\times\R^d}
  \Delta^2_x\psi(x,y)K(t,x,y)\,dxdy&
\nonumber\\
-\frac {4p}{p+1}\sum_{\substack{\mu,\nu,\iota=1}}^N\gamma_{\mu\nu}\int_{\R^d\times\R^d}
  |u_\mu(x)|^{p+1}|u_\nu(x)|^{p+1}m_{u_\iota}(y)\Delta_x\psi(x,y)\,dxdy,&
  \nonumber\\
\end{align}
with $\gamma_{\mu\nu}=\beta_{\mu\nu}+N\lambda_{\mu\nu}$, for any $\mu,\nu=1,\dots N$ and where 
\begin{align}\label{eq.kernel}
  K(t,x,y)\\
 =\nabla_xm_{u_\mu}(t,x)\cdot\nabla_y m_{u_\iota}(t,y)+\kappa m_{u_\mu}(t,x) m_{u_\iota}(t,y)+2m_{\nabla u_\mu}(t,x) m_{ u_\iota}(t,y).
 \nonumber
\end{align}
\end{lemma}

\begin{proof}
As before, we prove the identities for a smooth solution $(u_\nu)_{\nu=1}^N$, 
 switching to the general case $(u_\mu)_{\mu=1}^N \in\mathcal C(\R,H^2(\R^d)^N)$ by using a final density argument.
  First, one notices that \eqref{eq:intmor1},  thanks to the symmetry of the function $\psi(x,y)=\phi(|x-y|)$, it is equivalent to  
  \begin{align}\label{eq:mor0}
    \mathcal{ M}(t)&
    \\
    = \sum_{\mu,\iota=1}^N\int_{\R^d\times\R^d}
 m_{u_\iota}(y)j_{u_\mu}(x)\cdot\nabla_x\psi(x,y)+ m_{u_\mu}(x) j_{u_\iota}(y)\cdot\nabla_y\psi(x,y)\,dx\,dy,&
  \nonumber
  \end{align}
  Therefore, we differentiate w.r.t. time variable and get the identity
  \begin{align}
  \mathcal{\dot M}(t)
  =
 - \sum_{\mu,\iota=1}^N
 \Re \int_{\R^d\times\R^d}
 m_{u_\iota}(y) i\partial_t (\overline u_{\mu}(x)\nabla_xu_{\mu}(x))\cdot\nabla_x\psi(x,y)\,dx\,dy
   \nonumber&
    \\
   -\sum_{\mu,\iota=1}^N\Re\int_{\R^d\times\R^d}
  m_{u_\mu}(x)  i\partial_t (\overline u_{\iota}(y)\nabla_yu_{\iota}(y))\cdot\nabla_y\psi(x,y)\,dx\,dy&
    \label{eq:1}
    \\
    -\sum_{\mu,\iota=1}^N \Re \int_{\R^d\times\R^d}
 i\partial_t m_{u_\mu}(x)   \overline u_{\iota}(y)\nabla_yu_{\iota}(y)\cdot\nabla_y\psi(x,y)\,dx\,dy&
 \nonumber\\
  -\sum_{\mu,\iota=1}^N \Re \int_{\R^d\times\R^d} i\partial_t m_{u_\iota}(y)  \overline u_{\mu}(x)\nabla_xu_{\mu}(x)\cdot\nabla_x\psi(x,y)\,dx\,dy&
 \nonumber\\
 :=\mathcal {I}+\mathcal {II}+\mathcal {III}.
    \nonumber
  \end{align}
  Then, by using \eqref{eq:mor2}, 
  the Fubini's Theorem and exploiting  again the
  symmetry of $\psi(x,y)$ we can write
  \begin{align}
  \mathcal {I}=
  -2\sum_{\mu,\iota=1}^N\int_{\R^d\times \R^d}
  m_{u_\mu}(x)m_{u_\iota}(y)\Delta^3_x\psi(x,y)\,dxdy&
   \label{eq:2iniz}
  \\
   -2\sum_{\mu,\iota=1}^N\int_{\R^d\times \R^d}
  (\kappa m_{u_\mu}(x)m_{u_\iota}(y)+2m_{\nabla u_\mu}(t,x) m_{ u_\iota}(t,y))\Delta^2_x\psi(x,y)\,dxdy&
  \nonumber
  \\
   \ \ \ -\frac{4p}{p+1}\sum_{\substack{\mu,\nu,\iota=1\\
  \mu\neq\nu}}^N\gamma_{\mu\nu}\int_{\R^d\times \R^d}
  |u_\mu(x)|^{p+1}|u_\nu(x)|^{p+1}m_{u_\iota}(y)\Delta_x\psi(x,y)\,dxdy,&
  \nonumber
  \end{align}
where the last line of the above \eqref{eq:2iniz} is the sum of all the terms that are consequence of
the nonlinear part of the equation, while the first and the second lines are the sums of terms associated to the linear part of the equation. 
Rearranging the r.h.s. of the first line in \eqref{eq:2iniz} we have
  \begin{equation}\label{eq:equivalence}
    \begin{split}
    &-\sum_{\mu,\iota=1}^N\int_{\R^d\times\R^d} m_{u_\mu}(t,x)m_{u_\iota}(t,y)\Delta_x^3\psi(x,y)\,dxdy\\
    =&\sum_{\mu,\iota=1}^N\int_{\R^d\times\R^d} m_{u_\mu}(t,x)m_{u_\iota}(t,y)\partial_{x_i}\partial_{y_i}\Delta_x^2\psi(x,y)\,dxdy\\
    =&\sum_{\mu,\iota=1}^N\int_{\R^d\times\R^d}\nabla_x m_{u_\mu}(t,x)\cdot\nabla_ym_{u_\iota}(t,y)\Delta_x^2\psi(x,y)\,dxdy,
    \end{split}
  \end{equation}
applying integration by parts (see, for example \cite{PlVe})
and observing that $\partial_{x_k}\psi=-\partial_{y_k}\psi.$
Lastly, we arrive at
\begin{align}\label{eq:2}
 \mathcal{I}= 2\sum_{\mu,\iota=1}^N\int_{\R^d\times \R^d}
  \Delta^2_x\psi(x,y) K(t,x,y)\,dxdy&
  \\
  -\frac{4p}{p+1}\sum_{\mu,\iota=1}^N\gamma_{\mu\mu}\int_{\R^d\times \R^d}
  |u_\mu(x)|^{2p+2}m_{u_\iota}(y)\Delta_x\psi(x,y)\,dxdy&
  \nonumber
  \\
   \ \ \ -\frac{4p}{p+1}\sum_{\substack{\mu,\nu,\iota=1\\
  \mu\neq\nu}}^N\gamma_{\mu\nu}\int_{\R^d\times \R^d}
  |u_\mu(x)|^{p+1}|u_\nu(x)|^{p+1}m_{u_\iota}(y)\Delta_x\psi(x,y)\,dxdy.&
  \nonumber
\end{align}
In addition, by \eqref{eq:mor2} and the Fubini's Theorem we  can set
  \begin{align}
   \mathcal{II} =
  &
4\sum_{\mu,\iota=1}^N\int_{\R^d\times \R^d}
  \nabla u_\mu(x)D^2_x(\Delta_x-\kappa)\psi(x,y)\nabla\overline u_\mu(x)
  m_{u_\iota}(y)\,dxdy
  \nonumber
   \\
  &
   +4\sum_{\mu,\iota=1}^N\int_{\R^d\times \R^d}
  m_{u_\mu}(x)\nabla u_\iota(y)D^2_y(\Delta_y-\kappa)\psi(x,y)\nabla\overline u_\iota(y)
  \,dxdy
  \\
&-16\sum_{\mu,\iota=1}^N\int_{\R^d\times \R^d}D^2 u_{\mu}(x)D^2\psi(x,y)D^2 \overline u_{\mu}(x)m_{u_\iota}(y)\,dxdy
\nonumber
  \\
  &
  -8\kappa\sum_{\mu,\iota=1}^N\int_{\R^d\times \R^d}
  j_{u_\mu}(x)D^2_{xy}\psi(x,y) j_{u_\iota}(y)\,dxdy,
  \nonumber
  \end{align}
here we used, at least at this stage, the symmetry of $D^2(\Delta-\kappa) \psi$ to take out the real part condition in the first two 
terms of the sum on the r.h.s. of the above identity.  Let us now focus on $\mathcal {II}.$ It can be expressed as 
 \begin{align}\label{eq:b1a}
 \mathcal{II}=\sum_{\substack{\mu, \iota=1}}^N A^{\mu\iota},&
\end{align}
where, for each $\mu, \iota=1,...,N$, the $A^{\mu\iota}$ term is defined by the identity
\begin{align}\label{eqb1aI}
A^{\mu\iota} =  \,
  8\int_{\R^d\times\R^d}
 m_{u_{\iota}}(y) \nabla_xu_{\mu}(x)D^2_x(\Delta_x-\kappa)\psi(x,y)\nabla_x\overline u_{\mu}(x)\,dxdy &\\
-16\int_{\R^d\times \R^d}D^2 u_{\mu}(x)D^2\psi(x,y)D^2 \overline u_{\mu}(x)m_{u_\iota}(y)\,dxdy&
\nonumber
  \\
-8\kappa\int_{\R^d\times\R^d}
  j_{u_{\mu}}(x)D^2_{xy}\psi(x,y) j_{u_{\iota}}(y)\,dxdy&
  \nonumber\\
    =\,8\int_{\R^d\times\R^d}
 m_{u_{\iota}}(y) \nabla_xu_{\mu}(x)D^2_x\Delta_x\psi(x,y)\nabla_x\overline u_{\mu}(x)\,dxdy &
 \nonumber\\
 -16\int_{\R^d\times \R^d}D^2 u_{\mu}(x)D^2\psi(x,y)D^2 \overline u_{\mu}(x)m_{u_\iota}(y)\,dxdy&
\nonumber
  \\
 -4\kappa\int_{\R^d\times\R^d}
  m_{u_{\mu}}(x)\nabla_yu_{\iota}(y)D^2_y\psi(x,y)\nabla_y\overline u_{\iota}(y)\,dxdy &\nonumber\\
  - 4\kappa\int_{\R^d\times\R^d}
 m_{u_{\iota}}(y) \nabla_xu_{\mu}(x)D^2_x\psi(x,y)\nabla_x\overline u_{\mu}(x)\,dxdy& \nonumber\\
  -8\kappa\int_{\R^d\times\R^d}
  j_{u_{\mu}}(x)D^2_{xy}\psi(x,y) j_{u_{\iota}}(y)\,dxdy&
  \nonumber\\
  = \mathcal{II}_1+ \mathcal{II}_2.&
  \nonumber
  \end{align}
 We start by dealing with
  \begin{align}\label{eqb1aI2}
 \mathcal{II}_1 =-4\kappa\int_{\R^d\times\R^d}
  m_{u_{\mu}}(x)\nabla_yu_{\iota}(y)D^2_y\psi(x,y)\nabla_y\overline u_{\iota}(y)\,dxdy &\\
  - 4\kappa\int_{\R^d\times\R^d}
 m_{u_{\iota}}(y) \nabla_xu_{\mu}(x)\Delta_x\psi(x,y)\nabla_x\overline u_{\mu}(x)\,dxdy& \nonumber\\
  -8\kappa\int_{\R^d\times\R^d}
  j_{u_{\mu}}(x)D^2_{xy}\psi(x,y) j_{u_{\iota}}(y)\,dxdy.&
  \nonumber
   \end{align}
  Again by means of $\partial_{x_j}\psi = -\partial_{y_j}\psi$, for all $j=1,\dots,n$, one can verifies that the r.h.s. of the identity  \eqref{eqb1aI2}  is equal to
   \begin{align}\label{eq:b1b}
  4\kappa\int_{\R^d\times\R^d} \nabla_{y}u_{\iota}(y)D^2_{xy}\phi(|x-y|)\nabla_{y}\overline u_{\iota}(y)|u_{\mu}(x)|^2\,dxdy&
  \\
  +
  4\kappa\int_{\R^d\times\R^d}  \nabla_{x}u_{\mu}(x)D^2_{xy}\phi(|x-y|)\nabla_{x}\overline u_{\mu}(x)|u_{\iota}(y)|^2\,dxdy&
  \nonumber
  \\
  -8\kappa\int_{\R^d\times\R^d}
  \Im(\overline u_{\mu}(x)\nabla_{x}u_{\mu}(x))D^2_{xy}\phi(|x-y|)
  \Im(\overline u_{\iota}(y)\nabla_{y}u_{\iota}(y))\,dxdy&
  \nonumber
  \end{align}
  and finally to
  \begin{align}\label{eq:b1b2}
  =-4\kappa\int_{\R^d\times\R^d}
 \left(H_{\mu\iota} D^2_{x}\phi(|x-y|)\overline{H}_{\mu\iota}+G_{\mu\iota} D^2_{x}\phi(|x-y|)\overline{G}_{\mu\iota}\right)\,dxdy,&  \end{align}
with
   \begin{align*}
    H_{\mu\iota} 
    &
    :=u_{\mu}(t,x)\nabla_{y}\overline{u_{\mu}(t,y)}+\nabla_{x}u_{\iota}(t,x)\overline{u_{\iota}(t,y)},
    \\
    G_{\mu\iota}
    &
    :=u_{\mu}(t,x)\nabla_{y}u_{\iota}(t,y)-\nabla_{x}u_{\mu}(t,x)u_{\iota}(t,y).
  \end{align*}
  Thus by gathering \eqref{eq:b1a}, \eqref{eq:b1b} and since $\phi$ is a convex function one get $\mathcal{II}_{1}\leq 0$. Furthermore, by the assumptions \eqref{eq.bbelow} and \eqref{eq.bbelow2} one proves
  
   \begin{align}\label{eqb1aI3}
 \mathcal{II}_2 =  8\int_{\R^d\times\R^d}
 m_{u_{\iota}}(y) \nabla_xu_{\mu}(x)D^2_x\Delta_x\psi(x,y)\nabla_x\overline u_{\mu}(x)\,dxdy.&
 \\
 -16\int_{\R^d\times \R^d}D^2 u_{\mu}(x)D^2\psi(x,y)D^2 \overline u_{\mu}(x)m_{u_\iota}(y)\,dxdy&
\nonumber
  \\
\leq -8C\int_{\R^d\times\R^d} (\rho_1(|x-y|)+2\rho_3(|x-y|))|\nabla_v^{\bot}u_{\mu}(x)|^2m_{u_{\iota}}(y)\,dxdy\leq 0.
 \nonumber
   \end{align}
 It remains to control the last term 
   \begin{align}\label{eq:b2aI}
 \mathcal{III}=-2\sum_{\mu,\iota=1}^N \Re \int_{\R^d\times\R^d}
 i\partial_t m_{u_\mu}(x)   \overline u_{\iota}(y)\nabla_yu_{\iota}(y)\cdot\nabla_y\psi(x,y)\,dx\,dy&
\\
  -2\sum_{\mu,\iota=1}^N \Re \int_{\R^d\times\R^d} i\partial_t m_{u_\iota}(y)  \overline u_{\mu}(x)\nabla_xu_{\mu}(x)\cdot\nabla_x\psi(x,y)\,dx\,dy&
 \nonumber\\
  :=\mathcal{III}_1+ \mathcal{III}_2.
  \nonumber
\end{align}
We handle only the first sum of integrals on the r.h.s. of the above equality \eqref{eq:b2aI} because the second one
can be faced with a similar strategy. We observe that
     \begin{align}\label{eq:b2aInew}
 \mathcal{III}_1= 2\sum_{\mu,\iota=1}^N \Re \int_{\R^d\times\R^d}
 i\partial_t m_{u_\mu}(x)   \overline u_{\iota}(y)\nabla_yu_{\iota}(y)\cdot\nabla_y\psi(x,y)\,dxdy&
\\
 = -2\sum_{\mu,\iota=1}^N \Re \int_{\R^d\times\R^d}\overline u_{\mu}(x)\Delta(\Delta-\kappa)u_{\mu}(x)\overline u_{\iota}(y)\nabla_yu_{\iota}(y)\cdot\nabla_y\psi(x,y)\,dx\,dy&
 \nonumber\\
 + 2\sum_{\mu,\iota=1}^N \Re \int_{\R^d\times\R^d}  u_{\mu}(x)\Delta(\Delta-\kappa)\overline u_{\mu}(x)\overline u_{\iota}(y)\nabla_yu_{\iota}(y)\cdot\nabla_y\psi(x,y)\,dx dy.&
 \nonumber
\end{align}
Before to continue we need to notice the following fact. If we set 
$$
F(x)=\overline u_{\mu}(x)\Delta(\Delta-\kappa)u_{\mu}(x)+u_{\mu}(x)\Delta(\Delta-\kappa)\overline u_{\mu}(x),
$$
we have that 
  
    \begin{align}\label{eq:b2a}
4\Re \int_{\R^d\times\R^d}
  F(x)\overline u_{\iota}(y)\nabla_y u_{\iota}(y)\cdot\nabla_y\psi(x,y) \,dxdy&
 \\
  = \int_{\R^d\times\R^d}
 (F(x)+\overline{F}(x)) (\overline u_{\iota}(y)\nabla_y u_{\iota}(y)+ u_{\iota}(y)\nabla_y \overline u_{\iota}(y))\cdot\nabla_y\psi(x,y)\,dxdy &
 \nonumber\\
  -\int_{\R^d\times\R^d}
 (F(x)-\overline{F}(x)) (\overline u_{\iota}(y)\nabla_y u_{\iota}(y)-u_{\iota}(y)\nabla_y \overline u_{\iota}(y))\cdot\nabla_y\psi(x,y)\,dxdy &  
 \nonumber\\
  =
2   \int_{\R^d\times\R^d}
 (F(x)u_{\iota}(y)\nabla_y \overline u_{\iota}(y)+\overline{F}(x)\overline u_{\iota}(y)\nabla_y u_{\iota}(y))\cdot\nabla_y\psi(x,y)\,dxdy &
 \nonumber\\
  =
 4\Re  \int_{\R^d\times\R^d}
F(x)u_{\iota}(y)\nabla_y \overline u_{\iota}(y)\cdot\nabla_y\psi(x,y)\,dxdy. &
 \nonumber
\end{align}
Furthermore we obtain also the following
 \begin{align}\label{eq:b2a2}
2\Re \int_{\R^d\times\R^d}
  F(x) |u_{\iota}(y)|^2\Delta_y\psi(x,y) \,dxdy&
 \\
  = \int_{\R^d\times\R^d}
 (F(x)+\overline{F}(x)) (\overline u_{\iota}(y)\nabla_y u_{\iota}(y)+ u_{\iota}(y)\nabla_y \overline u_{\iota}(y))\cdot\nabla_y\psi(x,y)\,dxdy &
 \nonumber\\
  =
 \int_{\R^d\times\R^d}
 (F(x)\overline u_{\iota}(y)\nabla_y  u_{\iota}(y)+\overline{F}(x) u_{\iota}(y)\nabla_y \overline u_{\iota}(y))\cdot\nabla_y\psi(x,y)\,dxdy &
 \nonumber\\
 +   \int_{\R^d\times\R^d}
 (F(x)u_{\iota}(y)\nabla_y \overline u_{\iota}(y)+\overline{F}(x)\overline u_{\iota}(y)\nabla_y u_{\iota}(y))\cdot\nabla_y\psi(x,y)\,dxdy &
 \nonumber\\
  =
 2\Re \int_{\R^d\times\R^d}
  F(x)\overline u_{\iota}(y)\nabla_y u_{\iota}(y)\cdot\nabla_y\psi(x,y) \,dxdy&
  \nonumber\\
 +   2\Re\int_{\R^d\times\R^d}
 F(x)u_{\iota}(y)\nabla_y \overline u_{\iota}(y)\cdot\nabla_y\psi(x,y)\,dxdy. &
 \nonumber
\end{align}
Coupling \eqref{eq:b2a} and \eqref{eq:b2a2} we arrive at the equality
\begin{align}\label{eq:b2f}
  \Re \int_{\R^d\times\R^d}
  F(x) |u_{\iota}(y)|^2\Delta_y\psi(x,y) \,dxdy&
  \\
  =2\Re \int_{\R^d\times\R^d}
  F(x)\overline u_{\iota}(y)\nabla_y u_{\iota}(y)\cdot\nabla_y\psi(x,y) \,dxdy.&
  \nonumber
   \end{align}
  Then an application of  \eqref{eq:b2f} gives that the term on the r.h.s. of \eqref{eq:b2aInew} is equal to
 \begin{equation*}
 \begin{split}
 -2\Re \int_{\R^d\times\R^d}
 \overline u_{\mu}(x)\Delta(\Delta-\kappa)u_{\mu}(x) |u_{\iota}(y)|^2\Delta_y\psi(x,y)\,dxdy\\
 + 2\Re \int_{\R^d\times\R^d} u_{\mu}(x)\Delta(\Delta-\kappa)\overline u_{\mu}(x))|u_{\iota}(y)|^2\Delta_y\psi(x,y) \,dxdy=0.
 \end{split}
 \end{equation*}
 Thus one achieve $\mathcal{III}_1=0$ and for the same reasons, $\mathcal{III}_2=0$ also.
 Collecting all the previous steps we have that $\mathcal{II}+\mathcal{III}\leq 0$, which in turn, in combination with
\eqref{eq:1} and \eqref{eq:2}, implies \eqref{eq:intmor2} with $K(t,x,y)$ as in \eqref{eq.kernel}.
 \end{proof}
 
We have an equivalent version of the \eqref{eq:intmor2} which is interesting on its own because of the fact that the influence of the operator $-\kappa \Delta$ totally disappears. This is contained in the following
 
\begin{corollary}\label{cor:intmorbil}
 Let $(u_\mu)_{\mu=1}^N \in\mathcal C(\R, H^2(\R^d)^N)$, and $\psi=\psi(x,y)$ be as in Lemma \ref{lem:intmor}, then the following holds
\begin{align}
  \mathcal{ \dot M}(t) \label{eq:intmor2E}
   \leq
  2\sum_{\mu,\iota=1}^N\int_{\R^d\times\R^d}
  \Delta_x\psi(x,y)\widetilde K(t,x,y)\,dxdy&
 \\
-\frac {4p}{p+1}\sum_{\substack{\mu,\nu,\iota=1}}^N\gamma_{\mu\nu}\int_{\R^d\times\R^d}
  |u_\mu(x)|^{p+1}|u_\nu(x)|^{p+1}m_{u_\iota}(y)\Delta_x\psi(x,y)\,dxdy,&
  \nonumber
  \end{align}
 with $\gamma_{\mu\nu}=\beta_{\mu\nu}+N\lambda_{\mu\nu}$, for any $\mu,\nu=1,\dots N$ and where 
\begin{align}\label{eq.kernel2}
 \widetilde K(t,x,y)
 =-\Delta_xm_{u_\mu}(t,x)\Delta_y m_{u_\iota}(t,y)-2\nabla_xm_{\nabla u_\mu}(t,x) \nabla_ym_{ u_\iota}(t,y).
 \end{align}
\end{corollary}
\begin{proof}
 In \eqref{eq:2iniz}, we replace \eqref{eq:equivalence} by the identity
 \begin{equation}\label{eq:equivalence2}
    \begin{split}
    &-\sum_{\mu,\iota=1}^N\int_{\R^d\times\R^d} m_{u_\mu}(t,x)m_{u_\iota}(t,y)\Delta_x^3\psi(x,y)\,dxdy\\
    =-&\sum_{\mu,\iota=1}^N\int_{\R^d\times\R^d} m_{u_\mu}(t,x)m_{u_\iota}(t,y)\Delta_y\Delta_x\Delta_x\psi(x,y)\,dxdy\\
    =-&\sum_{\mu,\iota=1}^N\int_{\R^d\times\R^d}\Delta_x m_{u_\mu}(t,x)\Delta_ym_{u_\iota}(t,y)\Delta_x\psi(x,y)\,dxdy.
    \end{split}
  \end{equation}
We can achieve in addition 
\begin{align} \label{eq:intmorBE}
  2\int_{\R^d\times\R^d}
  \Delta^2_x\psi(x,y)m_{\nabla u_\mu}(t,x)m_{ u_\iota}(t,y)\,dxdy&
  \\
  =- 2\int_{\R^d\times\R^d}
   \nabla_x\cdot\nabla_y\Delta_x\psi(x,y)m_{\nabla u_\mu}(t,x)m_{ u_\iota}(t,y)\,dxdy&
  \nonumber\\
=- 2\int_{\R^d\times\R^d}
  \Delta_x\psi(x,y)\nabla_xm_{\nabla u_\mu}(t,x)\cdot\nabla_ym_{ u_\iota}(t,y)\,dxdy.&
   \nonumber
  \end{align}
We get also
\begin{align} \label{eq:intmorAE}
  \kappa \int_{\R^d\times\R^d}
  \Delta^2_x\psi(x,y) m_{u_\mu}(t,x) m_{u_\iota}(t,y)\,dxdy&
  \\
  =
- \kappa\int_{\R^d\times\R^d}
 D^2_{xy} D^2_x\psi(x,y) m_{u_\mu}(t,x) m_{u_\iota}(t,y)\,dxdy&
\nonumber\\
   =
  - \kappa\int_{\R^d\times\R^d}
 \nabla_xm_{u_\mu}(t,x)D^2_x\psi(x,y)\nabla_y m_{u_\iota}(t,y)\,dxdy&.
  \nonumber
  \end{align}
  Furthermore, one notices that the following relation is fulfilled 
  \begin{equation}\label{eq:intmorAEN}
\begin{split}
 - 2\kappa\int_{\R^d\times\R^d}
 \nabla_xm_{u_\mu}(t,x)D^2_x\psi(x,y)\nabla_y m_{u_\iota}(t,y)\,dxdy+\mathcal{II}_1 \\
  = -2 \kappa\int_{\R^d\times\R^d}
 \nabla_xm_{u_\mu}(t,x)D^2_x\psi(x,y)\nabla_y m_{u_\iota}(t,y)\,dxdy\\
 -4\kappa\int_{\R^d\times\R^d}
 \left(H_{\mu\iota} D^2_{x}\phi(|x-y|)\overline{H}_{\mu\iota}+G_{\mu\iota} D^2_{x}\phi(|x-y|)\overline{G}_{\mu\iota}\right)\,dxdy\\
= -4\kappa\int_{\R^d\times\R^d}
 H_{\mu\iota} D^2_{x}\phi(|x-y|)\overline{H}_{\mu\iota}\,dxdy\leq 0,\\
 \end{split}
\end{equation}
with $\mathcal{II}_1$ as in \eqref{eqb1aI2}. Then the identities \eqref{eq:equivalence2}, \eqref{eq:intmorBE} and \eqref{eq:intmorAE} in conjunction with \eqref{eq:intmorAEN} enable us to rewrite \eqref{eq:intmor2} as \eqref{eq:intmor2E} with $\widetilde K(t,x,y)$ as in \eqref{eq.kernel2}.

\end{proof}

We need to recall now that Lemma 2.1. in \cite{PaX} and Theorem 2.1. in \cite{GS} (see also \cite{Ca} for the general theory) in connection with the defocusing feature of the system imply a well-known result concerning global well-posedness 
for \eqref{eq:nls}. That is
\begin{proposition}\label{ConsLaw}
  Let $1\leq d \leq 8$ and $p>0$ be such that \eqref{eq:base} holds.
  Then for all $(u_{\mu,0})_{\mu=1}^N \in \Hin^2_x$ there exists a unique 
$(u_\mu)_{\mu=1}^N \in C(\R,\Hin^2_x) $ 
  solution to \eqref{eq:nls}, moreover 
\begin{align}
  &M(u_\mu)(t)=\norma{u_\mu(0)}_{L^2_x} \quad \text{ for all }\mu=1,\dots,N, \label{eq:massconservation}\\
  &E(u_1(t),\dots,u_N(t))=E(u_1(0),\dots,u_N(0))\label{eq:energyconservation},
\end{align}
with $E(u_1(t),\dots,u_N(t))$ as in \eqref{eq:energy}.
\end{proposition}

A direct consequence of Proposition \ref{ConsLaw}, Lemma \ref{lem:intmor} (and of Corollary \ref{cor:intmorbil}) is the following result concerning both the linear and nonlinear Morawetz estimates. Specifically we have

\begin{proposition}\label{dim3}
  Let $d\geq3$, $N\geq1$, $p>0$ be such that \eqref{eq:base} holds
  and let $(u_\mu)^N_{\mu=1}\in\mathcal C(\R,H^2(\R^d)^N)$ be a global solution to \eqref{eq:nls}. Then, if we indicate by $\gamma_{\mu\nu}=\beta_{\mu\nu}+N\lambda_{\mu\nu}$, for any $\mu,\nu=1,\dots N$, one has:

\begin{itemize}
\item for $d=3$:
\begin{align}\label{eq:stima}
  \int_{\R}\int_{\R}|\sum_{\mu=1}^Nm_{u_\mu}(t,x)|^{2}\,dt\,dx
+  \int_{\R}\int_{\R}|\sum_{\mu=1}^N\nabla_xm_{u_\mu}(t,x)|^{2}\,dt\,dx &
\\
+\sum_{\mu=1}^N \gamma_{\mu\mu}\int_{\R}\int_{\R}|u_\mu(t,x)|^{2p+4}\,dt\,dx\leq C\sum_{\mu=1}^N\|u_{\mu,0}\|^4_{H^2_x};&
\nonumber
\end{align}
 \end{itemize}

\begin{itemize}
\item for $d\geq 4$:
\begin{align}\label{eq:stima0}
  \sum_{\mu=1}^N\int_{\R}\int_{\R^d\times\R^d}\frac{\nabla_x |u_\mu(t,x)|^{2}\nabla_y|u_\mu(t,y)|^2}{|x-y|^3}
  \,dx\,dy\,dt&\\
 + \sum_{\mu=1}^N \gamma_{\mu\mu}\int_{\R}\int_{\R^d\times\R^d}\frac{ |u_\mu(t,x)|^{2p+2}|u_\mu(t,y)|^2}{|x-y|}
  \,dx\,dy\,dt
 \leq C\sum_{\mu=1}^N\|u_{\mu,0}\|^4_{H^2_x}. &
 \nonumber
 \end{align}
\end{itemize}
 
  \end{proposition}
\begin{proof}

 Let us choose  $\psi(x,y)=|x-y|$. After an easy calculation one achieve that
 
 \begin{equation}\label{eq.deltaH0}
\Delta_x|x-y|=
\frac{(n-1)}{|x-y|},
\end{equation}
for $d>1$ and
\begin{equation}\label{eq.deltaH}
\Delta^2_x|x-y|=
\begin{cases}
-\frac{(n-1)(n-3)}{|x-y|^3}  \ \ \ \  \  \  \text{if} \ \ \ d\geq4,\\
\\
-4\pi\delta_{x=y} \ \ \  \, \, \, \, \, \,  \, \, \, \, \, \, \, \, \text{if} \ \ \ d=3.
\end{cases}
\end{equation}
Additionally, it satisfies for $d\geq2$ and with $v=x-y$,
\begin{equation*}
D^2_x u(x)D^2_x|x-y|D^2_x \bar u(x)\geq 
\frac{(n-1)}{|x-y|^3}|\nabla_v^{\bot}u(x)|^2,
\end{equation*}
that is, the bound \eqref{eq.bbelow2} (we remand to \cite{LS}) and
\begin{equation*}
\nabla_x u(x)D_x^2\Delta^2_x|x-y|\nabla_x \bar u(x)
=-\frac{(n-1)}{|x-y|^3}(|\nabla_v^{\bot}u(x)|^2- 2|\nabla u(x)-\nabla_v^{\bot}u(x)|^2),
\end{equation*}
that is the bound \eqref{eq.bbelow} (here we remand to \cite{MWZ}).
From the inequality \eqref{eq:intmor2} with $\Delta^2\psi(x,y)\leq 0$ as in \eqref{eq.deltaH} and dropping one nonpositive term, we obtain
\begin{align}\label{eq:intmor3}
  \mathcal{ \dot M}(t) 
   \leq
  2 \sum_{\mu,\iota=1}^N \int_{\R^d\times\R^d}\nabla_xm_{u_\mu}(t,x)\cdot\nabla_y m_{u_\iota}(t,y)\Delta^2_x\psi(x,y)\,dxdy&
  \\
  +
  2 \kappa\sum_{\mu,\iota=1}^N \int_{\R^d\times\R^d} m_{u_\mu}(t,x) m_{u_\iota}(t,y)\Delta^2_x\psi(x,y)\,dxdy&
 \nonumber \\
  - \frac{4p}{p+1}\sum_{\mu,\iota=1}^N\gamma_{\mu\mu}\int_{\R^d\times\R^d}
  |u_\mu(t,x)|^{2p+2}m_{u_\iota}(t,y)\Delta_x\psi(x,y)\,dxdy&
  \nonumber
  \\
  -\frac{4p}{p+1}\sum_{\substack{\mu,\nu,\iota=1\\
  \mu\neq\nu}}^N\gamma_{\mu\nu}\int_{\R^d\times\R^d}
  |u_\mu(t,x)|^{p+1}|u_\nu(t,x)|^{p+1}m_{u_\iota}(t,y)\Delta_x\psi(x,y)\,dxdy,&
  \nonumber
\end{align}
then the first of the terms in the r.h.s. of the above equality can be written as
in Proposition \ref{dim3}, namely
\begin{align}\label{eq:intmor3fir}
 2 \sum_{\mu,\iota=1}^N \int_{\R^d\times\R^d}\nabla_xm_{u_\mu}(t,x)\nabla_y m_{u_\iota}(t,y)\Delta^2_x\psi(x,y)\,dxdy&
  \\
  =2\int_{\R^d\times\R^d}\sum_{\mu=1}^N\nabla_xm_{u_\mu}(t,x)\sum_{\iota=1}^N\nabla_y m_{u_\iota}(t,y)\Delta^2_x\psi(x,y)\,dxdy&
\nonumber  \\
=-2\int_{\R^d\times\R^d}\nabla_x\zeta(t,x)\nabla_y \zeta(t,y)\Delta^2_x\psi(x,y),
    \nonumber
\end{align}
with 
\begin{equation}\label{eq.zeta}
\zeta(t, \cdot)=\sum_{\mu=1}^N m_{u_\mu}(t,\cdot).
\end{equation}
By a direct inspection, one can check by using the Fourier transform and Plancherel's identity that
  \begin{align}\label{eq.int1}
\int_{\R^d\times\R^d}\nabla_x\zeta(t, x)\cdot\nabla_y \zeta(t, y)\Delta^2_x\psi(x,y)\,dxdy&\\
=-(\nabla \zeta(t, \cdot),(-\Delta)^{\frac{3-d}2} \nabla \zeta(t, \cdot))\leq 0,
\nonumber
\end{align}
where $(\cdot,\cdot)$ is the inner product in $L^2$. This means that the terms in the r.h.s. of the inequality \eqref{eq:intmor3} are all nonpositive. Integrating \eqref{eq:intmor3}  w.r.t. time variable over the interval $[T_1, T_2]$, for any $T_1, T_2\in \R$, one obtains by \eqref{eq:intmor1}
\begin{align}\label{eq:intmor3t}
 \sup_{t\in[T_1, T_2]} |\mathcal{ M}(t)| \\
   \geq
  -2 \sum_{\mu,\iota=1}^N \int_{T_1}^{T_2} \int_{\R^d\times\R^d}\nabla_xm_{u_\mu}(t,x)\cdot\nabla_y m_{u_\iota}(t,y)\Delta^2_x\psi(x,y)\,dxdy&
  \nonumber\\
  -2 \kappa\sum_{\mu,\iota=1}^N\int_{T_1}^{T_2} \int_{\R^d\times\R^d} m_{u_\mu}(t,x) m_{u_\iota}(t,y)\Delta^2_x\psi(x,y)\,dxdy&
 \nonumber \\
  + \frac{4p}{p+1}\left(\sum_{\mu,\iota=1}^N\gamma_{\mu\nu}\int_{T_1}^{T_2}\int_{\R^d\times\R^d}
  |u_\mu(t,x)|^{2p+2}m_{u_\iota}(t,y)\Delta_x\psi(x,y)\,dxdy\right.&
  \nonumber
  \\
 +\left. \sum_{\substack{\mu,\nu,\iota=1\\
  \mu\neq\nu}}^N\gamma_{\mu\nu}\int_{T_1}^{T_2}\int_{\R^d\times\R^d}
  |u_\mu(t,x)|^{p+1}|u_\nu(t,x)|^{p+1}m_{u_\iota}(t,y)\Delta_x\psi(x,y)\,dxdy\right),&
  \nonumber
\end{align}
where all the term on the r.h.s. of the above bound are nonnegative. We have also the following
important eastimate
\begin{align}
&
 2 \sup_{t\in[T_1, T_2]}\sum_{\mu,\iota=1}^N \left|\int_{\R^3}\int_{\R^3}j_{u_\mu}(t,x)\cdot\nabla_x\psi(x,y)m_{u_\iota}(t,y)\,dxdy\right|
\label{eq:a}
\\
&
\leq C_1\sup_{t\in[T_1, T_2]}\sum_{\mu=1}^N\|u_\mu(t)\|^4_{H^2_x}
\leq C_2\sum_{\mu=1}^N\|u_{\mu,0}\|^4_{H^2_x}<\infty,
\nonumber
\end{align}
for some $C_1,C_2>0$ and any $T_1,T_2\in \R$, because of the $H^2_x$-norm is conserved.  Thus \eqref{eq:stima} and \eqref{eq:stima0} follow by \eqref{eq:intmor3t}, and \eqref{eq:a}, letting $T_2\to\infty,  T_1\to-\infty$.
\end{proof}

It is interesting to see, as an alternate take, how we can arrive at the same results of the Proposition \ref{dim3} by using the
inequality \eqref{eq:intmor2E}. This is contained in the following

\begin{remark}\label{dim34}
 We can carry out \eqref{eq:stima} and \eqref{eq:stima0} by a direct use of the inequality \eqref{eq:intmor2E}. Pick up once again $\psi(x,y)=|x-y|$, we can manage the sum of integrals
\begin{equation}\label{eq:intmor3fint}
\begin{split}
-2\sum_{\mu,\iota=1}^N\int_{\R^d\times\R^d}
  \Delta_x\psi(x,y) \Delta_xm_{u_\mu}(t,x)\Delta_y m_{u_\iota}(t,y)\,dxdy\\
  -4\sum_{\mu,\iota=1}^N\int_{\R^d\times\R^d}
  \Delta_x\psi(x,y) \nabla_xm_{\nabla u_\mu}(t,x) \nabla_ym_{ u_\iota}(t,y)\,dxdy,
  \end{split}
\end{equation}
in two steps. The first of the terms in \eqref{eq:intmor3fint} can be reshaped in a similar way as in the proof
od Proposition \ref{dim3}, that is
\begin{align}\label{eq:intmor3fir2}
 - 2 \sum_{\mu,\iota=1}^N \int_{\R^d\times\R^d}\Delta_xm_{u_\mu}(t,x)\Delta_y m_{u_\iota}(t,y)\Delta_x\psi(x,y)\,dxdy&
  \\
  =-2\int_{\R^d\times\R^d}\sum_{\mu=1}^N\Delta_xm_{u_\mu}(t,x)\sum_{\iota=1}^N\Delta_y m_{u_\iota}(t,y)\Delta_x\psi(x,y)\,dxdy&
\nonumber  \\
=-2\int_{\R^d\times\R^d}\Delta_x\zeta(t,x)\Delta_y \zeta(t,y)\Delta_x\psi(x,y),
    \nonumber
\end{align}
with $\zeta(t, \cdot)$ as in \eqref{eq.zeta}. Proceeding as formerly, by a further use of the Fourier transform and Plancherel's identity, we check that
  \begin{align}\label{eq.int1a}
\int_{\R^d\times\R^d}\Delta_x\zeta(t, x)\cdot\nabla_y \zeta(x, y)\Delta_x\psi(x,y)\,dxdy&\\
=(\Delta\zeta(t,\cdot),(-\Delta)^{\frac{1-d}2} \Delta \zeta(t, \cdot))\geq 0.
\nonumber
\end{align}
The second term in \eqref{eq:intmor3fint} is easy to control. By bearing in mind that 
 \begin{align}\label{eq.L2prod}
\int_{\R^d\times\R^d}\nabla_xm_{\nabla u_\mu}(t,x)\cdot\nabla_y m_{u_\iota}(t,y)\Delta_x\psi(x,y)\,dxdy&&\\
=( \nabla m_{\nabla u_{\mu}}(t,\cdot),(-\Delta)^{\frac{1-d}2} \nabla m_{u_\iota}(t, \cdot))\geq 0,
\nonumber
\end{align}
for any $\mu, \iota= 1,\dots, N$ and $d\geq3$, then we have that
\begin{align}\label{eq:intmor3sec}
    -4 \sum_{\mu,\iota=1}^N \int_{\R^d\times\R^d}\nabla_xm_{\nabla u_\mu}(t,x)\cdot\nabla_y m_{u_\iota}(t,y)\Delta_x\psi(x,y)\,dxdy\leq 0.
    \end{align}
A combined use of \eqref{eq.int1a}, \eqref{eq.L2prod} and the argument in the proof of Proposition \ref{dim3}, imply \eqref{eq:stima} and \eqref{eq:stima0}.
\end{remark}
By \eqref{eq:intmor3t} in Proposition \ref{dim3} and the above Remark \ref{dim34} one arrives at the following corollary, where some new linear correlation-type estimates associated to the solution to \eqref{eq:nls} are obtained for $d=3,4$. In particular, for $d=4$, we get  (both for single and N-system equations) a similar estimates given in \cite{CT}, which is a diagonal, nonlinear analogue for NL4S of the bilinear refinement of Strichartz appeared in \cite{B98}. We have then

\begin{corollary}\label{correl}
Let $d\geq3$, $N\geq 1$, $p>0$ be such that \eqref{eq:base} holds
  and let $(u_\mu)^N_{\mu=1}\in\mathcal C(\R,H^2(\R^d)^N)$ be a global solution to \eqref{eq:nls}. Then one has,
  for $d\geq5$,
    \begin{align*}
  \sum^N_{\mu=1}\|(-\Delta)^{\frac{5-d}{4}}|u_\mu(t, x)|^2\|^2_{L^2((T_1, T_2);L_x^2)}\lesssim  \sup_{t\in[T_1, T_2]} |\mathcal{ M}(t)|.
  \end{align*}
In particular the following estimates are valid:
\begin{itemize}
\item for $d=3$, 
\begin{align*}
  \sum^N_{\mu=1}\|(-\Delta)^{\frac{1}{2}}|u_\mu(t, x)|^2\|^2_{L^2((T_1, T_2);L_x^2)}\lesssim  \sup_{t\in[T_1, T_2]} |\mathcal{ M}(t)|;
  \end{align*}
  \item  for $d=4$
  \begin{align*}
  \sum^N_{\mu=1}\|(-\Delta)^{\frac{1}{4}}|u_\mu(t, x)|^2\|^2_{L^2((T_1, T_2);L_x^2)}\lesssim  \sup_{t\in[T_1, T_2]} |\mathcal{ M}(t)|.
  \end{align*}
   \end{itemize}
\end{corollary}

\section{Proof of Theorem \ref{thm:main}}\label{MainThm}

The proof of Theorem \ref{thm:main} is divided in two steps. In the first one we 
shall exploit, inspired by the technicalities of \cite{CT} and \cite{Vis}, some decaying properties of the solution to NL4S \eqref{eq:nls}. In the second one we exhibit the proof of the
scattering by a combination of the argument recovered in the first step with the theory 
established for NLS in \cite{Ca} and \cite{GiVel}, here settled to the case of the system of NL4S. 
Beside this section we adopt the following notations: for any two positive real numbers $a, b,$ we write $a\lesssim b$  (resp. $a\gtrsim b$)
to indicate $a\leq C b$ (resp. $Ca\geq b$), with $C>0,$ we spread out the constant only when it is essential.
Moreover we shall set $w(t,x)=(u_\mu(t,x))_{\mu=1}^N,$ using both the notations where it is required. 

 \subsection{Decay of solutions to \eqref{eq:nls}}\label{decsol}
The purpose of this section is to show some decaying behaviour of the solution to  \eqref{eq:nls} compulsory in the proof of the scattering. One has the following
\begin{proposition}\label{decay}
Let $3\leq d \leq 8$ and  $p\in \R$ such that \eqref{eq:base} holds. If $w\in \mathcal C(\R,\Hin^2_x),$ 
is a global solution to \eqref{eq:nls}, then we have the decay property
\begin{align}\label{eq:decay1}
\lim_{t\rightarrow \pm \infty} \|w(t)\|_{\Lin^q_x}=0,
\end{align}
with $2<q<\frac{2d}{d-4},$ for $d\geq 5$ and with $2<q<+\infty,$ for $3\leq d\leq 4.$ 
\end{proposition}

\begin{proof}
We discuss only the case $t\rightarrow \infty$,
the case of $t\rightarrow -\infty$ can be treated similarly. Also we deal first with $d\geq 3,$ and then the case $d=3.$ 
Following \cite{GiVel2},
it is enough to prove the property \eqref{eq:decay1} 
for an appropriate $2<q<\frac{2d}{d-4},$ for $d\geq 5$ (and with $2<q<+\infty,$ for $3\leq d\leq 4$), since the outcome for the general case can be established by the combined action of conservation of mass \eqref{eq:massconservation},  kinetic energy in \eqref{eq:energyconservation}  
and interpolation. We want to prove that
\begin{equation}\label{eq:potenergy2}
\lim_{t\rightarrow \pm \infty} \|w(t)\|_{\Lin^{\frac{2d+4}{d}}_x}=0.
\end{equation}
With the purpose of doing that we proceed as in \cite{Vis}
assuming by the absurd that there exists a sequence
$\{t_n\}$ such that $t_n \to +\infty$ and
\begin{equation}\label{eq:sequencetime}
 \inf_n \|w(t_n, x)\|_{\Lin^{\frac{2d+4}{d}}_x}=\epsilon_0>0.
\end{equation}
Next, one recalls the localized Gagliardo-Nirenberg inequality provided in \cite{CT} 
(see also \cite{L1} and \cite{L2}):
\begin{equation}\label{eq:GNloc}
\|\varphi\|_{\Lin^{\frac{2d+4}{d}}_x}^{\frac{2d+4}{d}}\leq C \left(\sup_{x\in \R^d} \|\varphi\|_{\Lin^2(Q_x)}\right)^{\frac{4}{d}} 
\|\varphi\|^2_{\Hin^2_x},
\end{equation}
with $Q_x$ being the unit cube in $\R^d$ centered in $x$. By \eqref{eq:sequencetime}, \eqref{eq:GNloc}, in which we selected $\varphi=w(t_n, x)$, accomplishing the bound $\|w(t_n, x)\|_{\Hin^2_x}<+\infty$,
we argue that there exists $x_n \in \R^d$ such that
\begin{equation}\label{eq:sequencespace}
\|w(t_n, x)\|_{\Lin^2(Q_{x_n})}=\delta_0>0.
\end{equation}
We claim now that there exists $ \bar t>0$ such that
\begin{equation}\label{eq:claim0}
  \|w(t, x)\|_{\Lin^2(\widetilde Q_{x_n})}\geq \delta_0/2,
\end{equation}
for all  $t\in (t_n, t_n+\bar t)$ and where $\widetilde Q_x=x+[-2,2]^d$ denotes the cube in $\R^d$ of sidelenght $2$ centered at $x$.
To show \eqref{eq:claim0}
we fix a cut--off function $\widetilde\chi(x)\in C^\infty_0(\R^d)$, so as
$\widetilde\chi(x)=1$ for $x\in Q_x$ and $\widetilde\chi(x)=0$ for $x\notin \widetilde Q_x$.
Then one gets
\begin{equation*}
\begin{split}
\left|\frac d{dt} \int_{\R^d} \widetilde\chi(x -x_n)  | w(t,x)|^2 dx \right|\lesssim\kappa\left|\int_{\R^d} \Delta_x\widetilde\chi(x -x_n) \Im(\nabla_xw(t,x)\overline w(t,x))dx\right|\\
+\left|\int_{\R^d} \Delta_x\widetilde\chi(x -x_n) \Im(\Delta_xw(t,x)\overline w(t,x))dx\right|+\\
+\left|\int_{\R^d} \nabla_x\widetilde\chi(x -x_n) \Im(\Delta_xw(t,x)\nabla_x\overline w(t,x)) dx\right|
\lesssim \sup_t \|w(t,x)\|_{\Hin^2_x}^2.
\end{split}
\end{equation*}
By the fact that $H^2_x$-norm of the solution is preserved, we can figure out applying the fundamental theorem of calculus 

\begin{equation}
\left|\int_{\R^d} \widetilde\chi(x -x_n) (|w(s, x)|^2  - |w(t, x)|^2) dx\right|\leq C_1 |t-s|,
\end{equation}
for some $C_1>0$ that does not depend on $n$. For this reason, by picking up $t=t_n$, we get
the immediate inequality 

\begin{equation}
\int_{\R^d}\widetilde \chi(x -x_n) |w(s, x)|^2 dx+ C_1|t_n-s|\geq  \int_{\R^d} \chi(x -x_n) |w(t_n, x)|^2 dx,
\end{equation}
which yields, thanks to the characteristics of the function $\widetilde\chi$,

\begin{equation}
\int_{\widetilde Q_{x_n}} |w(s, x)|^2 dx\geq  \int_{Q_{x_n}} |w(t_n, x)|^2 dx - C_1|t_n-s|.
\end{equation}
Then \eqref{eq:claim0} follows as soon as we choose $\bar t>0$
such that
$3\delta_0^2>4C \bar t$.
However, \eqref{eq:claim0} is in contradiction with the Morawetz estimate \eqref{eq:stima0}.
Indeed, the above bound \eqref{eq:claim0} is such that
\begin{equation}\label{eq:claim1}
 \sum_{\mu=1}^N \|u_\mu(t)\|^{2}_{L^{2}_x(\widetilde Q_{x_n})} \geq C(d)\delta^{2}_0>0,
\end{equation}
for any $t\in (t_n, t_n+\bar t)$ with $\bar t$ as above and with these time intervals singled out to be disjoint.
As a consequence, by H\"older inequality, there exists $\bar\mu \in \{1,\dots,N\}$ so as to ensure
\begin{equation}\label{eq:lowb1}
 \|u_{\bar\mu}(t)\|^{\bar p}_{L^{\bar p}_x(\widetilde Q_{x_n})}\gtrsim \delta^{2}_0, 
\end{equation}
for all $\bar p\geq 2$ and where $t\in (t_n, t_n+\bar t)$, with $\bar t$ being again as above. Thus one concludes  that 
\begin{equation} \label{eq:stima3}
\begin{split}
  \min_{\mu=1,\dots,N} \gamma_{\mu\mu}& \sum_{\mu=1}^N \int_{\R}\int_{\R^d\times\R^d}\frac{ |u_\mu(t,x)|^{2p+2}|u_\mu(t,y)|^2}{|x-y|}
  \,dx\,dy\,dt \\
   &\gtrsim\sum_{\mu=1}^N \sum_{n} \int_{t_n}^{t_n+\bar t}\int_{\widetilde Q_{x_n}\times\widetilde Q_{x_n}} 
   |u_\mu(t,x)|^{2p+2}|u_\mu(t,y)|^2\,dx\,dy\,dt
	\\
 &\gtrsim   \sum_{n} \int_{t_n}^{t_n+\bar t} \delta^{4}_0\,dt=\infty,
\end{split}
\end{equation}
where in the last line we applied \eqref{eq:claim0} in conjunction with \eqref{eq:claim1}, \eqref{eq:lowb1}
and Fubini's Theorem. This produces a contradiction with  \eqref{eq:stima0}.\\

For $d=3$ we can proceed as in the previous case just using \eqref{eq:stima} instead of \eqref{eq:stima0}
 and arguing as above we arrive at 
\begin{equation} \label{eq:stima4}
  \begin{split}
  \min_{\mu=1,\dots,N} &\,\gamma_{\mu\mu} \sum_{\mu=1}^N  \int_{\R}\int_{\R^3}|u_\mu(t,x)|^{2p+4}\,dt\,dx \\
  &\gtrsim  \sum_{\mu=1}^N \sum_{n} \int_{t_n}^{t_n+\bar t}\int_{\widetilde Q_{x_n}} |u_\mu(t,x)|^{2p+4}\,dt\,dx
    =\infty,
  \end{split}
\end{equation}
 again the above inequality gives rise to a contradiction with the interaction estimate \eqref{eq:stima}. The proof is now completed.
\end{proof}

 \subsection{Scattering for the NL4S system \eqref{eq:nls}.}\label{NLSscat}
In this section we perform the proof of Theorem \ref{thm:main} which can be obtained following the classic theory (see \cite{Ca}, \cite{GiVel2} and references therein), however we present the results in a self-contained way adapted to the more general form of the systems regime. One can recall from \cite{Pa1}, \cite{PaX} (see also \cite{KT} and reference therein), the following 

\begin{definition}\label{Sadm}
An exponent pair $(q, r)$ is biharmonic-admissible if $2\leq q,r\leq \infty,$  $(q, r, n ) \neq (2,\infty, 4),$ and
\begin{align}\label{StrTV}
\frac 4q +\frac n{r}=\frac n2.
\end{align}
\end{definition} 

\begin{proposition}\label{Stri}
Let be two biharmonic-admissible pairs $(q,r)$, $(\widetilde q,\widetilde r)$.  Indicate by $\mathcal D=\nabla_{x}$ and by $\mathcal D^2=\Delta_{x}$
Then we have for $k=0,1, 2$ and $\kappa=0,1$ the following estimates:
\begin{align}
\label{eq:200p1}
 \| \mathcal D^k e^{ -it(\Delta^2_x-\kappa \Delta_x)} f\|_{L^q_t L^r_x} + \left \|\mathcal D^k\int_0^t e^{- i (t-\tau) (\Delta^2_x-\kappa \Delta_x)} F(\tau)
d\tau\right\|_{L^q_t L^r_x}&\\\nonumber
\leq C\big (\|\mathcal D^k f\|_{ L^2_x}+\|\mathcal D^k F\|_{L^{\widetilde q'}_t
L^{\widetilde r'}_x}\big ).&
 \end{align}
 \end{proposition}
Hence, for proving Theorem \ref{thm:main},  the following lemma is compelled to gain the space-time summability required for the scattering, that is
 \begin{lemma}\label{StriSys}
 Let us assume $p$ as in \eqref{eq:base}.
 Then, for any $w\in\mathcal C(\R,\Hin^2_x)$ global solution to \eqref{eq:nls}, we have
 \begin{align}
 w\in L^q(\R, \Win^{2,r}_x),
 \end{align}
for every biharmonic-admissible pair $(q,r)$.
\end{lemma}
\begin{proof}
We take account of the integral operator associated to \eqref{eq:nls} 
\begin{align}\label{eq:opint}
w(t)=e^ {it (\Delta^2_x-\kappa \Delta_x)}w_0 +  \int_{0}^{t} e^ {i(t-\tau) (\Delta^2_x-\kappa \Delta_x)} g(u(\tau),v(\tau),p) d\tau
\end{align}
where $t>0$ and 
\begin{gather}
w(t)=\begin{pmatrix}
      u_1(t)\\
      \vdots\\
      u_N(t)
      \end{pmatrix}, \quad
w_{0}= \begin{pmatrix}
     u_{1,0}\\
      \vdots\\
     u_{N,0}
     \end{pmatrix},
     \nonumber \\
   g(w,p)=
\begin{pmatrix}  \sum^N_{\nu=1}\beta_{1\nu} |u_1|^{p+1}|u_1|^{p-1}u_1\\
\vdots
\\ \sum^{N}_{\nu=1}\beta_{N\nu} |v_\nu|^{p+1}|u_1|^{p-1}u_N \end{pmatrix} +   \sum^N_{\mu,\nu=1}\lambda_{\mu\nu} |u_\mu|^{p+1}|u_\nu|^{p-1}u_\mu\begin{pmatrix}  1\\
\vdots
\\ 1 \end{pmatrix}.
\label{eq.NNL}
\end{gather}
We obtained the thesis by an use of the Strichartz estimates (see again \cite{Pa1} and \cite{Pa2}, for instance) which reduces to handle the inhomogeneous part in \eqref{eq:opint}. Moreover, we concentrate on the second term in \eqref{eq.NNL} only; the first one can be faced in the same manner. 
Select $( q', r')$ such that 
\begin{equation}\label{eq:pair}
 ( q, r):= \left(\frac{8(p+1)}{np},2p+2\right).
\end{equation}
The H\"older inequality and the Leibniz fractional rule give
\begin{equation}\label{eq.1nonl}
\begin{split}
 \|g(w, p)\|_{L^{q'}_{t>T}W^{2,r'}_x}
\lesssim\big\Vert \sum^N_{\mu,\nu=1}\lambda_{\mu\nu}\|u_\mu\|_{W^{2,r}_x}\|u_\nu|^{p+1}|u_\mu|^{p-1}\|_{L_x^{\frac{r}{2p}}}\big\Vert_{L^{q'}_{t>T}}&\\
\lesssim\sum^N_{\mu,\nu=1}\big\Vert  \|u_\mu\|_{W^{2,r}_x}\|u_\nu|^{p+1}|u_\mu|^{p-1}\|_{L_x^{\frac{r}{2p}}}\big\Vert_{L^{q'}_{t>T}},&
\end{split}
\end{equation}
By the inequality (see for instance \cite{HLP})
\begin{align*}
|u_\nu(x)|^{p+1}|u_\mu(x)|^{p-1}+|u_\mu(x)|^{p+1}|u_\nu(x)|^{p-1}\lesssim \left(|u_\mu(x)|^{2p}+|u_\nu(x)|^{2p}\right),
 \end{align*}
 one can see that the last term of the previous inequality is not exceeding 
  \begin{equation}\label{eq.1non2f}
 \begin{split}
     \sum_{\mu, \nu=1}^N\Big\|
   \|u_\mu\|_{W^{2,r}_x} \|u_\nu\|_{L_x^{r}}^{2p}  
   \Big \|_{L^{q'}_{t>T}}&
  \\
 \lesssim \sum^N_{\mu,\nu=1} \Big\|
   \|u_\mu\|_{W^{2,r}_x}  
   \Big(\|u_\nu\|_{L_x^{r}}^{2p(1-\theta)}\|u_\nu\|_{L_x^{r}}^{2p\theta}\Big)\|_{L^{q'}_{t>T}}.&
\end{split}
\end{equation}
We select now $\theta\in(0,1)$ so that $\theta=(q-q')/2pq'$, this implies that the term in the last line of  
\eqref{eq.1non2f} is controlled by
 \begin{equation}\label{eq.1non2s}
 \begin{split}
  \Big\| 
  \|w\|_{W^{2,r}_x} 
   \|w\|_{L_x^{r}}^{2p(1-\theta)}
  \|w\|_{L_x^{r}}^{2p\theta}  \Big\|_{L^{q'}_{t>T}}
    \lesssim \Big\| 
  \|w\|_{W^{2,r}_x} 
   \|w\|_{L_x^{r}}^{\frac{q}{q'}-1}
  \|w\|_{L_x^{r}}^{2p+1-\frac{q}{q'}}  \Big\|_{L^{q'}_{t>T}}&
  \\
  \lesssim \Big\|
   \|w\|_{W^{2,r}_x}^{\frac{q}{q'}}
   \|w\|_{L_x^{2p+2}}^{2p+1-\frac{q}{q'}}
   \Big\|_{L^{q'}_{t>T}}\lesssim \|w\|_{L^{\infty}_{t>T}L_x^{r}}^{2p+1-\frac{q}{q'}}
  \|w\|_{L^{q}_{t>T}W^{2,r}_x}^{q-1}&,
\end{split}
\end{equation}
with all the constants independent from $t,T$. These conclusions in association with the equation \eqref{eq:opint}, Proposition \ref{decay},
and an use of the inhomogenehouse Strichartz estimates in \eqref{eq:200p1}
lead to
\begin{align}\label{eq.1non4}
\|w\|_{L^{q}_{t>T} \Win^{1,r}_x}
\leq C\|w_0\|_{\Hin^2_x}+\epsilon(T)\|w\|_{L^{q}_{t>T}\Win^{2,r}_x}^{q-1},
\end{align}
where $\epsilon(T)\rightarrow 0$ as $T\rightarrow \infty.$ Then for $T$ sufficiently large we arrive at
$$
\|w\|_{L^{q}((T,t),\Win^{2,r}_x)}\leq \bar C,
$$
with the constant $\bar C$ independent from $t$. In that way we get that $w\in L^{q}((T,\infty), \Win^{2,r}_x).$ Analogously  we have $w\in L^{q}((-\infty, -T), \Win^{2,r}_x).$
We conclude by a continuity that $w\in L^q(\R, \Win^{2,r}_x)$.
\end{proof}

Another consequence the above lemma is the following:

\begin{proof}[Proof of Theorem \ref{thm:main}]
  The proof of Theorem \ref{thm:main} is now a straight consequence of Lemma \ref{StriSys} above:
  we shortly demonstrate it here for the sake of completeness.
  \\
  {\em Asymptotic completeness:} We write  $\overline w(t)=e^ {-it(\Delta^2_x-\kappa \Delta_x)}w(t)$ getting
  \begin{align}\label{eq:opint1}
\overline w(t)=w_0 +i  \int_{0}^{t} e^ {-is (\Delta^2_x-\kappa \Delta_x)}  g(w,p) ds,
\end{align}
  moreover one has, for $0<t_1<t_2$,
  \begin{align}\label{eq:opint2}
\overline w(t_2)-\overline w(t_1)=  i\int_{t_1}^{t_2} e^ {-is (\Delta^2_x-\kappa \Delta_x)}  g(w,p)ds.
\end{align}
By applying the Strichartz estimates \eqref{eq:200p1}, we infer to
 \begin{align}\label{eq:stric}
\|\overline w(t_1)-\overline w(t_2)\|_{\Hin^2_x}\lesssim &
\\
\|e^ {it\Delta_{x}} (\overline w(t_1)-\overline w(t_2))\|_{\Hin^2_x}\lesssim \| g(w,p)\|_{L^{q'}_{(t_1,t_2)} \Win^{2,r'}_x},&
\nonumber
\end{align}
with $(q, r)$ is Schr\"odinger-admissible biharmonic pair as in \eqref{eq:pair}. By the steps we followed in the proof of Lemma 
\ref{StriSys} one attains
 \begin{align}\label{eq:stric1}
\lim_{t_1,t_2\rightarrow \infty}\|\overline w(t_1)-\overline w(t_2)\|_{\Hin^2_x}=0.
\nonumber
\end{align}
Then we can see that there exists $(u_{1,0}^{\pm},\dotsc, u_{N,0}^\pm)\in H^2(\R^d)^N$ and thus the map
 $(u_1(t),\dotsc, u_N(t))\rightarrow (u_{1,0}^{\pm},\dotsc, u_{N,0}^\pm)$ in $H^2(\R^d)^N$ as $t\rightarrow\pm \infty.$
Note that, by Proposition \ref{ConsLaw}, we achieve also the following properties
\begin{equation}\label{eq:asco1}
\begin{split}
M(u_{1,0}^{\pm},\dotsc, u_{N,0}^\pm)=\|(u_{1,0},\dotsc, u_{N,0})\|^2_{\Lin^2_x}, \\ 
 \sum_{\mu=1}^N  \int_{\R^d} \Big(\abs{\Delta{u_{\mu,0}^{\pm}}}^2+\kappa|\nabla u_{\mu,0}^{\pm}|\Big) dx=E(u_{1,0},\dotsc, u_{N,0}).
\end{split}
\end{equation}

  {\em Existence of wave operators:} The construction of the wave operators is standard and straightforward from 
  what we see above. So we skip.
  \end{proof}

{\it {\bf Acknowledgements:}}
The author would like to thank Dipartimento di Matematica, Univerist\`a di Pisa, where He was Visiting Professor during the period of preparation of this paper.

\end{document}